\documentclass[11pt]{article}

\usepackage{amssymb,amsthm,amsmath}
\usepackage{graphicx}
\usepackage{epsf}

\newtheorem{theorem}{Theorem}
\newtheorem{lemma}{Lemma}
\newtheorem{claim}{Claim}

\title{Covering Paths and Trees for Planar Grids}

\author{
Bal\'azs Keszegh\thanks{Research supported by Hungarian National Science Fund (OTKA), under grant PD 108406 and under grant NN 102029 (EUROGIGA project GraDR 10-EuroGIGA-OP-003) and the J\'anos Bolyai Research Scholarship of the Hungarian Academy of Sciences.}
\\ 
\small {\it Hungarian Academy of Sciences,
Alfr\'ed R\'enyi Institute}\\[-0.8ex]
\small {\it of Mathematics, P.O.B. 127, Budapest H-1364, Hungary}\\
}

\begin{document}

\maketitle

\begin{abstract}
Given a set of points in the plane, a \emph{covering path} is a polygonal path
that visits all the points. In this paper we consider covering paths of the vertices of an $n \times m$ grid. We show that the minimal number of segments of such a path is $2\min(n,m)-1$ except when we allow crossings and $n=m\ge 3$, in which case the minimal number of segments of such a path is $2\min(n,m)-2$, i.e., in this case we can save one segment. In fact we show that these are true even if we consider covering trees instead of paths.

These results extend previous works on axis-aligned covering paths of $n\times m$ grids and complement the recent study of covering paths for points in general position, in which case the problem becomes significantly harder and is still open.
\end{abstract}

\section{Introduction}  \label{sec:intro}
In this paper we study polygonal paths visiting a finite set of points
in the plane.
A \emph{spanning path} is a directed Hamiltonian path
drawn with straight line edges. Each edge in the path connects two
of the points, so a spanning path can only turn at one of the given
points. Every spanning path of a set of $n$ points consists of $n-1$
segments. A \emph{covering path} is a directed polygonal path in the
plane that visits all the points. A covering path can make a turn at
any point, i.e., either at one of the given points or at a (chosen) Steiner point.
Obviously, a spanning path for a point set $V$ is also a covering path for $V$.
If no three points in $V$ are collinear, every covering path consists of at
least $\lceil n/2\rceil$ segments. A \emph{minimum-segment} or \emph{minimum-link} covering path
for $V$ is one with the smallest number of segments (links).
A point set is said to be in \emph{general position} if no three
points are collinear.

As a generalization of the well-known puzzle of linking $9$ dots in a $3 \times 3$ grid with a
polygonal path having only $4$ segments~\cite{Lo1914}, Mori\'c~\cite{Mo11} proposed the study of covering paths for point sets in general position.
In this direction, in \cite{coveringpaths} the trivial lower and upper bounds of $n/2$ and $n-1$ where improved, if we allow crossings, then the upper bounds gets close to $n/2$:

\begin{theorem}\cite{coveringpaths}\label{T1}
Every set of $n$ points in the plane admits a (possibly self-crossing)
covering path consisting of $n/2 +O(n/\log{n})$ line segments.
\end{theorem}

However, for noncrossing paths none of the trivial bounds are asymptotically best:
\begin{theorem}\cite{coveringpaths} \label{paththm}
Every set of $n$ points in the plane admits a noncrossing covering path
with at most  $\lceil (1-1/601080391)n\rceil-1$ segments. There exist $n$-element point sets that
require at least $(5n-4)/9$ segments in any noncrossing covering path.
\end{theorem}

The problem considered in this paper is another natural generalization of the $3\times 3 $ puzzle, we consider covering paths of the vertices of an $n\times m$ grid. For simplicity, we only consider orthogonal equally spaced (unit distance) grids (which we simply call unit grids), although most of our results hold in more general settings. A non-unit grid denotes a not necessarily equally spaced orthogonal grid. Note that from now on by a grid we also mean the planar point set which is equal to the set of vertices of the grid. If we allow only axis-aligned paths (all segments are vertical or horizontal), then the minimal number of segments is $2n+1$ for an $n\times n$ grid and $2\min(n,m)-1$ for an $n\times m$ grid \cite{KKM94}. There are various other upper and lower bounds on the minimum number of segments needed in an axis-aligned path traversing an $n$-element point set in $\mathbb{R}^d$ ~\cite{BBD+08,CM98,Co04,KKM94}. For more related results we refer to the introduction of \cite{coveringpaths}.

We solve the case when we allow segments of any direction, the study of which was proposed by D. P\'alv\" olgyi \cite{dompc}. The results of this paper are the following.
\begin{theorem}\label{thmpath}
The minimum number of segments of a (possibly self-crossing) path covering an $n\times m$ grid is $2\min(n,m)-2$ if $n= m\ge 3$ and $2\min(n,m)-1$ otherwise. On the other hand, the minimum number of segments of a noncrossing covering path is always $2\min(n,m)-1$.
\end{theorem}

The upper bound constructions of Theorem \ref{thmpath} are not new results, the general upper bound $2\min(n,m)-1$ is rather trivial (and folklore) and for the case $n= m\ge 3$ see e.g. Puzzle $141.$ in \cite{puzzlebook}.

Theorem \ref{thmpath} shows that by using segments of any direction, we cannot win anything or we can win at most one segment compared to using only axis-aligned segments. We can also consider covering trees instead of covering paths, in which case we can count the number of edges of the tree or the number of segments of the tree. When counting the edges, it turns out that we still cannot do better than with paths, thus as a path is also a tree, we get a strengthening of Theorem \ref{thmpath}:

\begin{theorem}\label{thmtree}
The minimum number of edges of a (possibly self-crossing) tree covering an $n\times m$ grid is $2\min(n,m)-2$ if $n= m\ge 3$ and $2\min(n,m)-1$ otherwise. On the other hand, the minimum number of edges of a noncrossing covering tree is always $2\min(n,m)-1$.
\end{theorem}

However, if we count the number of segments of the tree, we can finally win something, as it is easy to see that we need exactly $\min(n,m)+1$ segments:

\begin{claim}\label{claimtree}
The minimum number of segments of a (possibly self-crossing or noncrossing) tree covering an $n\times m$ grid is $\min(n,m)+1$, except when $n=1$ or $m=1$ or $n=m=2$ in which cases it is $\min(n,m)$.
\end{claim}

Although for simplicity we considered only orthogonal unit grids, these results are almost true for arbitrary grids. Specifically, from our proofs it follows that even if the grid is non-orthogonal and non-unit (i.e., the grid-lines are not necessarily equally spaced), for noncrossing covering paths/trees the answer remains the same; for crossing paths/trees when $n= m\ge 3$ does not hold, the answer again remains the same; finally, if $n= m\ge 3$, then we may or may not win $1$ segment, depending on the point set, but we cannot win more than $1$ segment. Also, the minimum number of segments of a covering tree remains $\min(n,m)+1$.

The proofs of the results are contained in the next section.

\section{Proofs}\label{sec:proofs}
We denote the $n\times m$ orthogonal unit grid by $V(n,m)$, where $n$ is the number of columns and $m$ is the number of rows. An $n\times m$ orthogonal grid which is not necessarily a unit-grid is denoted by $V'(n,m)$. Theorem \ref{thmpath} and \ref{thmtree} follow from the following lemmas.

As we noted in Section \ref{sec:intro}, the first part of the following lemma is trivial and the second part also appears in the literature \cite{puzzlebook}. For self-containedness we include its short proof.
\begin{lemma}\label{lemma:constr}
There exists a noncrossing path with $2\min(n,m)-1$ segments covering $V(n,m)$. If $n= m\ge 3$ then there exists a (self-crossing) path with $2\min(n,m)-2$ segments covering $V(n,m)$.
\end{lemma}
\begin{proof}
For a noncrossing covering path with $2\min(n,m)-1$ there are several different constructions, we show an axis-aligned covering path. Wlog. assume $n\ge m$, for each row take a segment from its leftmost point to its rightmost point. Connect these segments to form a path by connecting every second point in the first and last columns (in the first column we connect the odd pairs, in the last column the even pairs). See Figure \ref{fig:worm}(a).

If $n= m\ge 3$, there are again several different constructions for a covering path with $2\min(n,m)-2$ segments. One possibility is extending the well-known solution for the $3\times 3$ grid to bigger grids by an axis-aligned spiral. Figure \ref{fig:spiral} shows this for the $5\times 5$ grid, from which the interested reader can easily construct it for bigger grids, we omit the exact definition of the construction.
\end{proof}

To prove the lower bounds, we first prove a lemma about covering a grid with non-axis-aligned segments (which not necessarily form a tree or path).

\begin{figure}
    \centering
    \includegraphics[height=4cm]{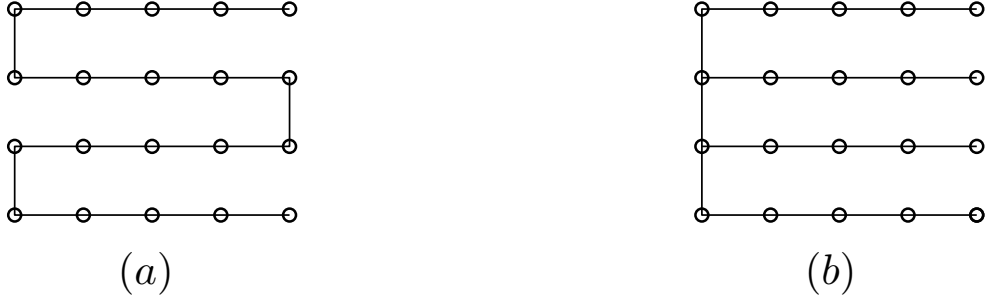}
   \caption{(a) noncrossing covering path with $2\min(n,m)-1$ segments (b) noncrossing covering tree with $\min(n,m)+1$ segments}
   \label{fig:worm}
\end{figure}

\begin{lemma}\label{lemma:nonaligned}
Any covering of a not necessarily unit grid $V'(k,l)$ with a set of segments, none of them horizontal nor vertical, has at least $k+l-2$ segments.
If $k=1$ or $l=1$ or the set of segments is noncrossing (segments may have common endpoints), then the set has at least $k+l-1$ segments.
\end{lemma}
\begin{proof}
If $k=1$ or $l=1$ then trivially a non-axis-aligned segment can cover at most one point from $V'(k,l)$ and thus we need $\max(k,l)=k+l-1$ segments.

If $k,l\ge 2$ then consider only the top and bottom rows and leftmost and rightmost columns, these contain altogether $2k+2l-4$ points. We say that these points are on the border of the grid. The intersection of these rows and columns define $4$ corner points. As any non-axis-aligned segment can cover at most two points of the border, we need at least $(2k+2l-4)/2=k+l-2$ segments to cover all of them. 

Further, suppose that the set of segments is noncrossing. If there exists a border point for which the segment covering this point does not cover any other point from the border, then we need at least $\lceil 1+(2k+2l-5)/2\rceil=k+l-1$ segments. We call such a point a wasteful point. Thus if the topleft corner $A$ is wasteful, we are done. Otherwise, $A$ is covered by a segment $a$ which covers another point $A'$. As $a$ is a non-axis-aligned segment, $A'$ is not the topright nor the bottomleft corner and $A'$ is either in the bottom row or in the rightmost column (or both, if $A'$ is the bottomright corner). We consider the first case, as the second is analogous. Thus $A'$ is in the bottom row and then the segment $c$ covering the bottomleft corner $C$ cannot contain any other point from the border as otherwise it would be either axis-aligned or would intersect $a$. Thus $C$ is wasteful and we are done.
%Further, suppose that the set of segments is noncrossing. If for all corners the segment covering this corner does not cover any other point, then we need at least $4+(2k+2l-4-4)=k+l$ segments. Thus we can suppose that one corner, wlog. the topleft corner $A$ is covered by a segment $a$ which covers another point $A'$. As the segment is non-axis-aligned, $A'$ is not the topright nor the bottomleft corner. $A'$ is either in the bottom row or in the rightmost column (or both, in case $A'$ is the bottomright corner). We consider the first case, as the second is analogous. Thus $A'$ is in the bottom row and then the segment $c$ covering the bottomleft corner $C$ cannot contain any other point from the border as otherwise it would be either axis-aligned or would intersect $a$. Now if the segment $d$ covering the bottomright corner $D$ does not cover any other point then we have two segments that cover only one point and thus at least $2+(2k+2l-4-2)=k+l-1$ segments altogether. If $d$ covers any other point $D'$ (note that $a=d$ if $A'=D')$, that must be in the top row. Similarly as before, then the segment $b$ covering the topright corner $B$ cannot contain any other point. Again, we have two segments that cover only one point and thus at least $k+l-1$ segments altogether.
\end{proof}

\begin{lemma}\label{lemma:bound}
Every tree $T$ covering $V(n,m)$ contains
\begin{itemize}
\item[] at least $2\min(n,m)-2$ edges,
\item[] at least $2\min(n,m)-1$ edges, if $n,m\le 2$,
\item[] at least $2\min(n,m)-1$ edges, if $n\ne m$,
\item[] at least $2\min(n,m)-1=2n-1$ edges, if $n=m\ge 3$ and the tree $T$ is noncrossing.
\end{itemize}
\end{lemma}

\begin{figure}
    \centering
    \includegraphics[height=4cm]{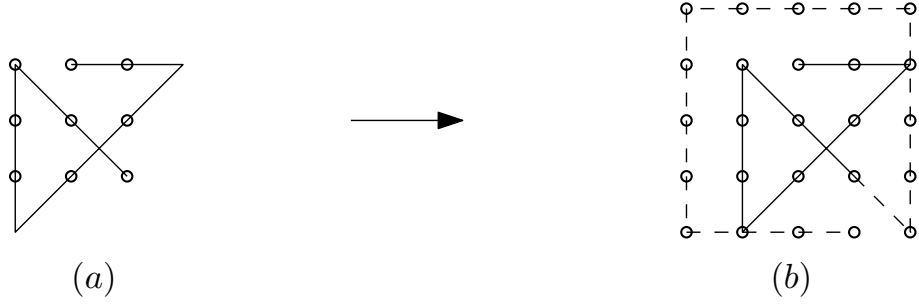}
   \caption{(self-crossing) covering path with $2n-2$ segments on an $n\times n$ grid for (a) $n=3$ and (b) $n=5$}
   \label{fig:spiral}
\end{figure}

\begin{proof}
The covering tree of $V(n,m)$ is denoted by $T$. 
It is an easy case analysis to check that if $n,m\le 2$ then every (possibly self-crossing) tree covering $V(n,m)$ contains at least $2\min(n,m)-1$ edges. 
First suppose that in all rows there exists a point which is covered by a horizontal segment of the tree (the point may be an endvertex of the tree).
As we have $m$ rows, we identified $m$ segments of the tree. As all these segments are parallel, they are independent in the tree, thus from the connectedness of the tree it follows that we have at least $m-1$ further edges, $2m-1$ altogether. 
Similarly, if in all columns there exists a point which is covered by a vertical segment of the tree, then we have at leat $2n-1$ edges altogether.
In both of these cases we have at least $2\min(n,m)-1$ edges, thus we can suppose that none of them holds.

Denote by $R_1$ the (nonempty) set of rows which do not contain a point which is covered by a horizontal segment and by $C_1$ the (nonempty) set of columns which do not contain a point which is covered by a vertical segment. Denote by $V'(k,l)$ the intersection of these rows and columns ($k=|C_1|,l=|R_1|$). This is a not necessarily unit grid, none of its points covered by axis-aligned segments. Thus we can apply Lemma \ref{lemma:nonaligned}, and we get that as $T$ also covers $V'(k,l)$, $T$ has at least $k+l-2$ non-axis-aligned segments. If $T$ is noncrossing then from the lemma we get that it contains at least $k+l-1$ non-axis-aligned segments. On the other hand, in each row not in $R_1$ $T$ has a horizontal segment, altogether $m-l$ horizontal segments. Similarly the columns not in $C_1$ contain $n-k$ vertical segments. Thus $T$ always has at least $k+l-2+m-l+n-k=n+m-2\ge 2\min(n,m)-2$ segments. Improving this, if $n\ne m$ then $T$ has at least $n+m-2\ge 2\min(n,m)-1$ segments. If $n=m$ and $T$ is noncrossing we again have at least $k+l-1+m-l+n-k=n+m-1\ge 2\min(n,m)-1$ segments. 
\end{proof}

Lemma \ref{lemma:constr} and Lemma \ref{lemma:bound} together imply Theorem \ref{thmpath} and Theorem \ref{thmtree}. 
What remains is to prove Claim \ref{claimtree}.

\begin{proof}[Proof of Claim \ref{claimtree}]
Recall that we want to prove, that the minimal number of segments of a tree covering $V(n,m)$ is $\min(n,m)+1$, except when $n=1$ or $m=1$ or $n=m=2$ in which cases it is $\min(n,m)$. A simple example of a tree covering $V(n,m)$ and having $\min(n,m)+1$ segments is a tree that (wlog. suppose $n\ge m$) contains the segment covering the leftmost row and for each column a segment that covers that column (see Figure \ref{fig:worm}(b)).

If $n=1$ or $m=1$ then $1$ segment is enough to cover $V(n,m)$. If $n,m\ge 2$ then we need at least $2$ segments to cover $V(n,m)$. If $n=m=2$, a covering tree with two segments consists of the two main diagonals, and one segment is trivially not enough to cover. From now on we suppose that we are not in any of these exceptional cases. It is easy to see that any segment can cover at most $\max(n,m)$ points. If at leasy one segment covers less than this many points then we have at least $nm/\max(n,m)+1=\min(n,m)+1$ segments. Thus we can suppose that all segments cover exactly $\max(n,m)$ points. It is easy to see that all of them is either axis-aligned or contained in one of the two main diagonal lines. We cannot cover only with segments in the main diagonal lines and we cannot cover only with one axis-aligned segment. Thus, as the tree is connected, there is an axis-aligned segment and a segment that intersects this segment, and these two segments must intersect in a point of $V(n,m)$.  The intersection point is covered by both of them, thus the two of them together only cover at most $2\max(n,m)-1$ points. From this it again follows that we have at least $\min(n,m)+1$ segments.
\end{proof}

\section{Open Problems}

In this paper we completely solved the planar case of this problem. As it was done for axis-aligned paths, it would also be interesting to consider covering paths (with no restriction on the direction of the segments) of higher dimensional grids, and of 3-dimensional grids in particular. If we want to cover an $n\times n\times n$ unit grid using an axis-aligned path, then we need $\frac{3}{2}n^2+O(n)$ segments \cite{BBD+08}. What happens if we allow the path to have segments of any direction? What happens if we relax the condition so that we want to cover with a tree (that has segments of any direction) instead of a path? What difference does it make if we allow or forbid crossings?

\paragraph{Acknowledgment.}
The author is grateful to A. Dumitrescu, D. P\'alv\" olgyi and G. Tardos for their helpful comments and observations.

\end{document}